\documentclass{amsart}[12pt]
\newcommand{\private}[1]{} 

\newcommand{\plfn}[1]{\private{%
{\footnote{{\bf Comm.Pascal:}{#1}}}}}

\usepackage{amssymb, amsmath, amscd, amsthm, color, epsfig}
\usepackage[all]{xy}
\newcommand{\refequ}[1]{$(\ref{#1})$}

\theoremstyle{plain}
\newtheorem{thm}{Theorem}[section]

\newtheorem{lemma}[thm]{Lemma}
\newtheorem{prop}[thm]{Proposition}

\theoremstyle{definition}
\newtheorem{defin}[thm]{Definition}

\theoremstyle{remark}
\newtheorem{rmk}[thm]{Remark}
\numberwithin{equation}{section}

\newcommand{\ie}{\emph{i.e.\ }}

\newcommand{\BQ}{\mathbb Q}
\newcommand{\BC}{\mathbb C}
\newcommand{\BZ}{\mathbb Z}
\newcommand{\Bk}{\mathbf {k}}
\newcommand{\calC}{\mathcal{C}}
\newcommand{\calK}{\mathcal{K}}
\newcommand{\barN}{\underline{N}}
\newcommand{\barA}{\underline{A}}
\newcommand{\barDelta}{\underline{\Delta}}
\newcommand{\barM}{\underline{M}}

\newcommand{\quism}{\stackrel{\simeq}{\rightarrow}}
\newcommand{\leftquism}{\stackrel{\simeq}{\leftarrow}}
\newcommand{\iso}{\stackrel{\cong}{\rightarrow}}

\newcommand{\Apl}{A_{PL}}
\newcommand{\homset}{{\operatorname{Hom}}}
\newcommand{\id}{{\mathrm{id}}}
\newcommand{\sgn}{{\mathrm{sgn}}}
\newcommand{\pos}{\mathrm{pos}}
\newcommand{\pr}{\mathrm{pr}}
\newcommand{\codim}{\mathrm{codim}}
\newcommand{\TotCof}{{\mathrm{TotCof}}}%
\newcommand{\Top}{\mathrm{Top}}%
\newcommand{\sSets}{\mathrm{sSets}}%
\newcommand{\CDGA}{\mathrm{CDGA}}
\newcommand{\DGMod}{\mathrm{DGmod}}
\newcommand{\hocolim}{\operatorname{hocolim}}

\newcommand{\set}[1]{\left\{{#1}\right\}}

\newcommand{\refN}[1]{$(\ref{#1})$}
\newcommand{\refD}[1]{Definition \ref{D:#1}}
\newcommand{\refS}[1]{Section \ref{S:#1}}
\newcommand{\refL}[1]{Lemma \ref{L:#1}}
\newcommand{\refT}[1]{Theorem \ref{T:#1}}
\newcommand{\refE}[1]{Equation $(\ref{E:#1})$}
\newcommand{\refP}[1]{Proposition \ref{P:#1}}
\begin{document}

\title[Models for configuration spaces]{A remarkable DG-module model for configuration spaces}
\author{Pascal Lambrechts}
\author{Don Stanley}%
\address{P.L.: Universit\'e Catholique de Louvain, Institut Math\'ematique\\
2, chemin du Cyclotron\\B-1348 Louvain-la-Neuve, BELGIUM}
\email{pascal.lambrechts@uclouvain.be}
\address{D.S.:
University of Regina. Saskatchewan. CANADA.
}%
\email{Donald.Stanley@uregina.ca}%
\thanks{The authors gratefully acknowledge support by the 
Institute Mittag-Leffler (Djursholm, Sweden). }
\thanks{P.L. is Chercheur Qualifi\'e au F.N.R.S.}%

\begin{abstract}
\private{{\Huge CONFIDENTIAL VERSION\\ONLY FOR COAUTHORS}\\
{\bf see the command ``private'' at the beginning of the tex file}}
Let $M$ be a simply-connected closed manifold and consider the (ordered)
configuration space of $k$ points in $M$, $F(M,k)$. 
In this paper we construct a commutative differential graded algebra which is a
potential candidate for a model of the rational homotopy type of $F(M,k)$. We prove that our model it is at least a
$\Sigma_k$-equivariant differential graded model.

We also study Lefschetz duality at the level of cochains and describe equivariant models of the complement
of a union of polyhedra in a closed manifold.
\end{abstract}

\maketitle


\section{Introduction}
\label{S:intro}

Let $M$ be a closed simply-connected triangulable manifold of
dimension $m$. The (ordered) configuration space of $k$ points in $M$ is the space
$$F(M,k):=\{(x_1,\cdots,x_k)\in M^k:x_i\not=x_j\textrm{ for }i\not=j\}.$$
An interesting problem is whether the homotopy type of that configuration space 
depends only on the homotopy type of $M$. Longoni and Salvatore
 \cite{LongSalv} have discovered an example of two homotopy equivalent manifolds whose configuration
 spaces of two points are not homotopy equivalent. Their examples are non-simply connected.
By contrast, a general position argument implies that for a $2$-connected closed manifold
 the configuration space of two points depends only on the homotopy type  of the
 manifold. More generally we have proved in \cite{LS:FM2Q} that the rational homotopy
 type of $F(M,2)$ depends only on the rational homotopy type of $M$, under the 
 $2$-connectivity hypothesis,
  and we have build an explicit model (in the sense of Sullivan) of that configuration space out
 of a model of $M$. 
 
 The goal of the present paper is to exhibit a promising candidate for the model of the rational 
 homotopy type of the $F(M,k)$. To explain this, first recall the Sullivan fu¬nctor
 $$\Apl\colon\Top\to \CDGA$$
 where $\CDGA$ is the category of commutative differential graded algebras. The main feature
of this functor is that the rational homotopy type of a simply-connected space of finite type, $X$,
is encoded in any CDGA quasi-isomorphic to $\Apl(X)$. Such a CDGA is called a \emph{CDGA-model}
 of $X$.
 
 In \cite{LS:PDCDGA-x2} we have proved that any simply-connected manifold $M$ admits a CDGA
 model, $(A,d)$, such that $A$ is a Poincar\'e duality algebra 
of dimension $m=\dim M$.
 We can then define a \emph{diagonal class}
$$\Delta:=\sum_\lambda(-1)^{\deg(a_\lambda)} a_\lambda\otimes a_\lambda^*\in A\otimes A$$
where $\{a_\lambda\}$ is a basis of $A$ and $\{a_\lambda^*\}$ is the Poincar\'e dual basis.
In the present paper we describe a CDGA 
\begin{equation}
\label{E:FAk}
F(A,k):=\left(\frac{A^{\otimes k}\otimes\operatorname{E}(g_{ij}:1\leq i<j\leq k)}{(\textrm{Arnold and symmetry relations})},d(g_{ij})=\pi^*_{ij}(\Delta)\right)
\end{equation}
where $\operatorname{E}(g_{ij})$ is an exterior algebra on 
${{k}\choose{2}}$ 
generators $g_{ij}$ of degree $m-1$, $\pi^*_i(a)=1^{\otimes i-1}\otimes a\otimes1^{\otimes k-i}\in A^{\otimes k}$ and
$\pi_{ij}^*(a\otimes b)=\pi^*_i(a){\cdot}\pi^*_j(b)$ (see \refD{FAk} for a complete definition.)

When $k=2$, $F(A,2)$ is weakly equivalent to the CDGA model of $F(M,2)$ built in \cite[Theorem 5.6]{LS:FM2Q}, 
and when $M$ is a complex projective variety then $F(H^*(M;\BQ);k)$ is equivalent to the Fulton-MacPherson-Kriz 
CDGA-model of $F(M,k)$
built in \cite{FM:CC} and \cite{Kriz:FMkQ}.
We are not able to prove in general  that for $k\geq 3$, $F(A,k)$ is a CDGA-model of $F(M,k)$ but at least we can
prove that it is an equivariant DGmodule model of it. More precisely the inclusion $F(M,k)\hookrightarrow M^k$
and Kunneth quasi-isomorphism induce an $\Apl(M)^{\otimes k}$-module structure on $\Apl(F(M,k))$.
Suppose given quasi-isomorphisms of CDGA, $A\leftquism R\quism\Apl(M)$. Our main result (\refT{main})
states that $\Apl(F(M,k))$ and $F(A,k)$ are weakly equivalent $R^{\otimes k}$-DGmodules, even $\Sigma_k$-equivarianly
where  $\Sigma_k$ is the symmetric group on $k$ letters acting by permutation of the factors.

Our proof goes through an ``equivariant cochain-level Lefschetz duality theorem for a system of subpolyhedra
in a closed manifold.'' In more detail, classical Lefschetz duality determines $H^*(W\smallsetminus X)$ from
the map $H_*(X)\to H_*(W)$ when $X$ is a subpolyhedron of a closed oriented manifold $W$.
 In \cite{LS:AlgPEmb} we have studied Lefschetz duality at the level of models instead of homology. 
In this paper we generalize this further by considering $X$ as a union of a finite family of subpolyhedra 
$\{X_e\hookrightarrow W\}_{e\in E}$. The idea is that Lefschetz duality gives a weak equivalence
between $C^*(W\smallsetminus\cup_{e\in E}X_e)$ and the mapping cone of the dual of the map
$C^*(W)\to C^*(\cup_{e\in E}X_e)$. On the other hand a generalized Mayer-Vietoris theorem gives
a weak equivalence between $C^*(\cup_{e\in E}X_e)$ and a chain complex built out
of the chain complexes $C^*(\cap_{e\in\gamma}X_e)$ for non empty subsets $\gamma\subset E$.
When a discrete group $G$ acts on the manifold $W$ preserving in a certain sense the
system $\{X_e\hookrightarrow W\}_{e\in E}$, all these weak equivalences can be choosen
to be equivariant.

This generalized Lefschetz duality can be applied  to the system of partial diagonals
$\Delta_{ij}=\{(x_1,\cdots,x_k)\in M^k:x_i=x_j\}$ so that
$F(M,k)=M^k\smallsetminus\cup_{1\leq i<j\leq k}\Delta_{ij}$.
This approach was already taken by Bendersky and Gitler in \cite{BenGit}.
The difference with their paper is that we apply Lefschetz duality at the level of models
in order to get a model of $F(M,k)$ instead of a model of
the pair $(M^k,\cup_{1\leq i<j\leq k}\Delta_{i,j})$ as they do.
Also we carefully study the action of the symmetric group on that model.

The model \refN{E:FAk} also  gives rise to a spectral sequence by filtering by the length 
in the variables $g_{ij}$. This spectral sequence coincides with the two spectral sequences
studied in \cite{FT:confMassey}.
In particular, as F\'elix and Thomas  show in that paper, this spectral sequence does not always 
collapse when $k\geq4$.
Also the fixed point CDGA, $F(A,k)^{\Sigma_k}$, is a DGmodule model of the unordered configuration space, and
 F\'elix and Tanr\'e proved in \cite{FT:Hunordconf} that the associated spectral sequence does collapse.

In the last section we explain how our approach could be useful to the study of other complement spaces,
such as complements of unions of projective subspaces in $\BC P(n)$.

Here is a plan of this paper. In \refS{notation} we recall some notation in particular for
the suspension and dual of DG-modules and for their mapping cones.
In \refS{FAk} we construct in detail the CDGA $F(A,k)$ associated to a Poincar\'e duality
CDGA $A$. In \refS{functorsC*} we introduce a variant $\widehat{C}^*$ of the cochain
algebra functor $C^*$ for which the excision quasi-isomorphism is actually an isomorphism.
In \refS{groupactionsmodules} we fix some notation  for the action of a discrete group on
DG-modules. In \refS{equivD} we establish an equivariant cochain level Lefschetz duality theorem giving a model 
of the complement $W\smallsetminus X$. 
In \refS{cubical}  we study, for a set of subpolyhedra $\{X_e\subset W\}_{e\in E}$, 
cubical diagrams like $\{\textrm{subsets }\gamma\subset E\}\to
\{C^*(\underset{e\in\gamma}\cap X_e)\}$ and define their \emph{total cofibres}
which will turn out to be models for the cochains on the  complement of the union of polyhedra, 
$C^*(W\smallsetminus\cup_{e\in E}X_e)$. In \refS{Gcubical} 
we put an action on the cubical diagrams and total cofibers from the previous section.
In \refS{equivLefschetzSystPolyh} we finally establish the equivariant cochain level Lefschetz duality 
for a system of polyhedra (\refT{EquivLefschSystem}).
In \refS{modelconf} we apply the above theory to prove that $F(A,k)$ is an equivariant DGmodule model
of $F(M,k)$. The last section is an informal discussion about other possible applications of this approach.¤

\subsection{Acknowledgment}
We thank  Soren Illmann for discussions on simplicial actions. We acknowledge support of  the institute Mittag-Leffler where
part of this research was done during a common
stay of the two authors.
\section{Notation}\label{S:notation}
In this short section we recall some standard notation.

We fix a ground field $\Bk$.
We will consider
non-commutative and commutative non negatively graded differential algebras,
or DGA and CDGA for short. The degrees are written as superscripts and the
differential increases the degree. If $R$ is a (C)DGA we will consider also right
differential graded modules over $R$  ($R$-DGmodules for short,
see \cite{FHT:RHT} for the precise definitions).

The $k$-th suspension of an $R$-dgmodule $M$ is the $R$-dgmodule
$s^kM$ defined by
\begin{itemize}
\item $(s^kM)^i=M^{i+k}$ as vector spaces,
\item $(s^kx){\cdot}r=s^k(x{\cdot}r)$ for $x\in M,r\in R$,
\item $d(s^kx)=(-1)^ks^k(dx)$ for $x\in M$.
\end{itemize}
Therefore $\deg(s^{k}x)=-k+\deg(x)$. We have a natural isomorphism $s^kM\otimes s^l N\cong
s^{k+l}(M\otimes N)$ sending $s^kx\otimes s^ly$ to $(-1)^{l\deg(x)}s^{k+l}(x\otimes y)$.

The dual of a graded vector space $V$ is the graded vector space $\#V$
defined by
$$(\#V)^k=\homset(V^{-k},\Bk).$$
If $M$ is a right $R$-dgmodule then $\#M$ inherits an obvious {\em left}
$R$-dgmodule structure. When $R$ is a commutative DGA, we
can turn $\#M$ into a right $R$-dgmodule structure
by the rule
$$
\phi{\cdot}r:=(-1)^{\deg(\phi){\cdot}\deg(r)}r{\cdot}\phi,
\mbox{\,\,for $r\in R$, $\phi\in\#M$.}
$$
We have canonical isomorphisms $s^{k}\#M\cong\#s^kM$, given by $(s^kf)(s^kx)=(-1)^{k\deg(f)}f(x)$,
and, under a finite type assumption,  $\#M\otimes\#N\cong\#(M\otimes N)$, given by 
$(f\otimes g)(x\otimes y)=(-1)^{\deg(g){\cdot}\deg(x)}f(x){\cdot}g(y)$.

If $f\colon (M,d_M)\to (N,d_N)$ is a morphism of $R$-dgmodule, the
{\em mapping cone} of $f$ is the $R$-dgmodule
$$C(f):=(N\oplus_fsM,d)$$
defined by
\begin{itemize}
\item $C(f)=N\oplus sM$ as $R$-module,
\item $d(y,sx)=(d_N(y)+f(x),-s(d_M(x)))$ for $x\in M,y\in N$.
\end{itemize}
\section{The CDGA $F(A,k)$}
\label{S:FAk}
An \emph{oriented Poincar\'e algebra of formal dimension $m$} is a
couple $(A,\omega)$ where $A$ is a graded commutative
$\Bk$-algebra and $\omega\colon A^m\to\Bk$ is a linear form such
that each pairing $A^i\otimes A^{m-i}\to\Bk,a\otimes
b\mapsto\omega(a.b)$ is non-degenerate. When $A$ is also equipped
with a differential that makes it a CDGA, the following
definition, which comes from \cite[Definition 4.6]{LS:FM2Q} or
\cite[Definition 2.2]{LS:PDCDGA-x2}, expresses the compatibility between
the Poincar\'e duality and the CDGA structures:
\begin{defin}\label{D:PD-CDGA}
An
\emph{oriented differential Poincar\'e duality algebra} or
\emph{oriented Poincar\'e duality CDGA} is a triple $(A,d,\omega)$
such that
\begin{itemize}
\item[(i)] $(A,d)$ is a CDGA;
\item[(ii)] $(A,\omega)$ is an oriented Poincar\'e duality
algebra of formal dimension $m$;
\item[(iii)] $\omega(dA)=0$.
\end{itemize}
\end{defin}

Note that when $(A,\omega)$ is a \emph{connected} Poincar\'e algebra and
$(A,d)$ is a CDGA such that the class of maximal degree represents
a non trivial homology class then  $(A,d,\omega)$ is a Poincar\'e duality
CDGA, as proved in \cite[Proposition 4.8]{LS:FM2Q}. The main result of
\cite{LS:PDCDGA-x2} states that any closed oriented simply-connected manifold
admits a CDGA-model which is a connected Poincar\'e duality CDGA.

Let $A$ be an oriented Poincar\'e duality CDGA of formal dimension
$m$. Next we recall the diagonal class $\Delta\in(A\otimes A)^m$
as defined in \cite[Definition 4.4]{LS:FM2Q}. Let
$\set{a_\lambda}_{0\leq\lambda\leq N}$ be a basis of $A$ and
$\set{a_\lambda^*}$ be its Poincar\'e dual basis with respect to
the orientation, that is
$\omega(a_\lambda{\cdot}a_\mu^*)=\delta_{\lambda\mu}$ where
$\delta_{\lambda\mu}$ is the Kronecker symbol. The diagonal
class is
\begin{equation}
\label{E:Delta}
\Delta:=\sum_\lambda(-1)^{\deg(a_\lambda)}a_\lambda\otimes
a_\lambda^*\in A\otimes A \end{equation}  It is proved in
\cite[Proposition 4.3, Proposition 4.11, and remark after
Definition 4.4]{LS:FM2Q} that $\Delta$ is a cocycle of degree $m$
which is independent of the choice of the basis. When $A$ is connected this diagonal class is also, up to a
scalar multiple, independent of the choice of the
orientation. See also \cite{Abrams:2DTQFTFrob}, where it is explained how a Poincar\'e duality algebra,
as a Frobenius algebra, becomes a coalgebra, and hence the diagonal class can also be seen as the coproduct of the 
orientation class.\plfn{I do not believe that thi remark is very relevant but the referee asked for it...}

Consider the CDGA $A^{\otimes k}=A\otimes\cdots\otimes A$. 
For $1\leq i\leq k$ we consider the CDGA maps
$$
\pi_i^*\colon A\to A^{\otimes k}\,,\,a\mapsto
\underbrace{1\otimes\cdots\otimes1}_{i-1}\otimes a \otimes
\underbrace{1\otimes\cdots\otimes1}_{k-i}
$$
and for $1\leq i<j\leq k$ the maps
$$
\pi_{ij}^*\colon A\otimes A\to A^{\otimes k}\,,\,a\otimes b\mapsto
\underbrace{1\otimes\cdots\otimes1}_{i-1}\otimes a \otimes
\underbrace{1\otimes\cdots\otimes1}_{j-i-1}\otimes b \otimes
\underbrace{1\otimes\cdots\otimes1}_{k-j},
$$
that is $\pi^*_{ij}(a\otimes b)=\pi_i^*(a){\cdot}\pi^*_j(b)$.

Consider the relative Sullivan algebra (\cite[\S 14]{FHT:RHT})
$$
\left(A^{\otimes k}\otimes\wedge(g_{ij}:1\leq i<j\leq k),d\right)
$$
with $\deg(g_{ij})=m-1$ and $d(g_{ij})=\pi^*_{ij}(\Delta)$. Notice
that $d^2=0$ because $\Delta$ is a cocycle.
By convention we set\plfn{The referee in remark 9 asks to delay this convention but actually we use it in the arnold relation just below as well as in
Definition 3.4(ii)}
$$g_{ji}=(-1)^{m}g_{ij}.$$
Let $I$ be the ideal of $A^{\otimes k}\otimes\wedge(g_{ij})$
generated by the following relations
\begin{itemize}
\item[(i)] the Arnold or three-terms relations
$$g_{ij}g_{jl}+g_{jl}g_{li}+g_{li}g_{ij}\quad\quad\textrm{for }1\leq i<j<l\leq k;$$
\item[(ii)] the symmetry relations
$$\left(\pi^*_i(a)-\pi^*_j(a)\right)g_{ij}\quad\quad\textrm{for }1\leq i<j\leq k\textrm{ and }a\in A;$$
\item[(iii)] $g_{ij}^2=0\quad\quad\textrm{for }1\leq i<j\leq k$.
\end{itemize}

\begin{lemma}\label{L:Iideal}
The ideal $I$ generated by (i)-(iii) above is a differential
ideal of $\left(A^{\otimes k}\otimes\wedge(g_{ij}),d\right)$.
\end{lemma}
\begin{proof}
One computes that
\begin{eqnarray*}
\left(d(g_{ij}-g_{il})\right){\cdot}g_{jl}
&=&\left(\pi^*_{ij}(\Delta)-\pi^*_{il}(\Delta)\right){\cdot}g_{jl}\\
&=&\sum_\lambda\pi^*_i(a_\lambda)\left(\pi^*_j(a^*_\lambda)-\pi^*_l(a^*_\lambda)\right)g_{jl}\in
I
\end{eqnarray*}
and this easily implies that
$d\left(g_{ij}g_{jl}+g_{jl}g_{li}+g_{li}g_{ij}\right)\in I$.

Using Poincar\'e duality, it is straightforward
to check that
$
(1\otimes a){\cdot}\Delta-(a\otimes 1){\cdot}\Delta=0$ in $A\otimes A
$
(see \cite[Lemma 4.5]{LS:FM2Q}). This implies that
$d\left(\left(\pi^*_i(a)-\pi^*_j(a)\right)g_{ij}\right)\in I$.

It remains to prove that $d(g_{ij}^2)\in I$. If $m$ is even
it is immediate by the Leibniz rule. If $m$ is odd we can choose a basis $\{a_\lambda\}_{0\leq \lambda\leq N}$
of $A$ such that $a^*_\lambda=a_{N-\lambda}$, hence 
$\Delta=\sum_{\lambda=0}^{(N-1)/2}(-1)^{|a_ \lambda |}(a_ \lambda\otimes a_{N-\lambda}-a_{N-\lambda}\otimes a_ \lambda)$.
From the symmetry relations (ii) one deduces easily that for $a,b\in A$ we have
$\pi_{ij}^*(a\otimes b-(-1)^{|a||b|}b\otimes a)g_{ij}\in I$.  Therefore $d(g_{ij}^2)=2\pi_{ij}^*(\Delta)g_{ij}\in I$. 
\end{proof}
\rmk The hypothesis that $A$ is a Poincar\'e duality CDGA is
essential for making $I$ a differential ideal, hence for $F(A,k)$ below to be a CDGA.

\begin{defin}\label{D:FAk}
Let $(A,d)$ 
be a Poincar\'e duality CDGA of formal dimension $m$. We
define the \emph{$k$-configuration CDGA} as
$$
F(A,k):=\left(\frac{A^{\otimes k}\otimes\wedge(g_{ij}:1\leq
i<j\leq k)}{I},d(g_{ij})=\pi^*_{ij}(\Delta)\right)
$$
where $g_{ij}$, $\Delta$, $\pi^*_{ij}$, and $I$ are defined as
above.
We equip this CDGA  with a left action of the symmetric group
$\Sigma_k$ on $k$ letters generated by
\begin{itemize}
\item[(i)]$\sigma{\cdot}\left(\pi^*_i(a)\right)=\pi^*_{\sigma(i)}(a)
\quad\quad\textrm{for }1\leq i\leq k,a\in A, \textrm{ and
}\sigma\in\Sigma_k;$
\item[(ii)]$\sigma{\cdot}g_{ij}=g_{\sigma(i)\sigma(j)}\quad\quad\quad\textrm{for }1\leq i<j\leq k\textrm{ and
}\sigma\in\Sigma_k.$
\end{itemize}
\end{defin}
When $A$ is connected, since  the diagonal class is independent of
the orientation (up to a scalar multiple), the CDGA $F(A,k)$ does
not depend on the choice of the orientation.

\section{The cochain functor $C^*$ and excision isomorphisms}
\label{S:functorsC*} In this paper we will consider mainly the two
following contravariant cochain functors
\begin{itemize}
\item the singular cochain functor with coefficients in a field $\Bk$,
$$S^*(-;\Bk)\colon\Top\to\Bk\!-\!DGA$$
where the algebra structure comes from the cup product defined
through the usual front face/back face formula;
\item the Sullivan  functor of piecewise polynomial forms
with coefficients in a field $\Bk$ of characteristic zero,
$$\Apl(-;\Bk)\colon\Top\to\Bk\!-\!CDGA$$
as defined in \cite{BG:RHT} or \cite[\S 10]{FHT:RHT}.
\end{itemize}
We will denote by $C^*$ either of the two functors $S^*(-;\Bk)$
or  $\Apl(-;\Bk)$. Notice that an element $\omega\in C^*(X)$ is
completely determined by its values
$\langle\omega,\sigma\rangle$  (which belong to $\Bk$ when $C^*=S^*(-;\Bk)$ and
to the CDGA  $(\Apl^*)_{\deg(\sigma)}$ defined in \cite[\S 10 (c)]{FHT:RHT} when  $C^*=\Apl$)
on singular simplices $\sigma$ in the singular simplicial set $S_\bullet(X)$.

The functor $C^*$ extends to pairs of topological spaces by 
 $C^*(X,A):=\ker(C^*(X)\to C^*(A))$.
If $(X,A)$ is a pair of topological spaces and if
$i\colon(X',A')\subset(X,A)$ is a subpair such that
$\overline{X\smallsetminus X'}\subset\mathrm{int}(A)$ and $A'=X'\cap A$
then the excision theorem implies that the restriction map
$$ C^*(i)\colon C^*(X,A)\quism C^*(X',A')$$
is a quasi-isomorphism. However $C^*(i)$ is almost never an
isomorphism. We show now how we can replace $C^*$ by a
quasi-isomorphic functor such that the morphism induced by $i$ is
indeed an isomorphism, at least on suitable triangulated pairs.
This will be usefull in our proof of \refT{equLefschetz} of equivariant cochain level
Lefschetz duality.

Let $\calK$ be the category of \emph{ordered simplicial complexes} defined as follows.%
\plfn{I
changed the category to \emph{ordered}
simplicial complexes because the simplicial set that I associated
to a (non ordered) simplicial complex in the previous version was actually not homotopy equivalent!!
This implies further changes below}
An (abstract) simplicial complex is a collection of finite non-empty sets, called \emph{simplices},
such that every non-empty subset of a simplex is also a simplex,
\cite[\S 3]{Munkres:AT}. The union of that collection is the set of \emph{vertices}.
An object of  $\calK$ is a simplicial complex with a partial order on the vertices such that
each simplex is linearly ordered. A morphism of $\calK$ is a simplicial map that preserves the
order of the vertices. We denote by $|K|$ the geometric realization of an ordered simplicial complex $K$.
 
Our goal is to build
a functor
$$
\widehat C^*:\calK\to\Bk\!-\!(C)DGA
$$
satisfying the two following properties \begin{itemize}
\item[(A)] there is a natural quasi-isomorphism of (C)DGA
$C^*(|-|)\quism \widehat C^*$, and
\item[(B)]  $\widehat C^*$
satisfies the following \emph{strict
excision statement}:\\
Let $(K,L)$ be a pair of ordered simplicial complexes, let
$K'\subset K$ be a subcomplex and set $L'=L\cap K$. If $K'\cup
L=K$ then the inclusion $i\colon (K',L')\subset(K,L)$ induces an
isomorphism
$$
\widehat C^*(i)\colon \widehat C^*(K,L)\stackrel{\cong}\to
\widehat C^*(K',L').
$$
where $\widehat C^*$ is extended to pairs by setting $\widehat C^*(K,L)=\ker(\widehat C^*(K)
\to\widehat C^*(L))$.
\end{itemize}

We treat separately the cases $S^*$ and $\Apl$. Suppose first that
$C^*=S^*$ are the singular cochains.  Let $\sSets$ be the category of simplicial sets.
 To an ordered simplicial complex $K$, one associates a simplicial set $K_\bullet$
 defined by
 \[K_n=\{(v_0,\dots,v_n):\{v_0,\dots,v_n\}\textrm{ is a simplex of $K$ and }
 v_0\leq\cdots\leq v_n\}
 \]
and the faces and degeneracies are defined by forgetting or repeating a vertex ,
see \cite[Example (1.3)]{Curtis:SHT} or \cite[Example 8.1.8]{Weibel:HA}.
We have a homeomorphism $|K|\cong|K_\bullet|$, see \cite[Exercise 8.1.4]{Weibel:HA}.

The normalized chain complex of the free simplicial abelian group $\BZ[K_\bullet]$
generated by $K_\bullet$ is isomorphic to the oriented chain complex of $K$ 
as defined in \cite[\S 5]{Munkres:AT}. We denote it by $N_*(\BZ[K_\bullet])$ and we consider
the dual cochain complex
\[\widehat C^*(K):=\homset(N_*(\BZ[K_\bullet]),\Bk)\]
which becomes a $\Bk$-DGA by defining a cup product through the usual front face/back face formula
as in \cite[\S 49]{Munkres:AT}. This defines a functor
\[
\widehat C^*:\calK\to\Bk\!-\!DGA
\]
and by \cite[Theorem 49.1]{Munkres:AT} we have a natural  quasi-isomorphism of DGA
\[C^*(|K|)=S^*(|K|;\Bk)\quism\widehat C^*(K).\]

We check that $\widehat C^*$ satisfies the strict excision statement.
Notice that an element $\phi\in \widehat C^*(K,L)$ is determined by its values
in $\Bk$ on the simplices of $K$.
Suppose that $i\colon (K',L')\hookrightarrow(K,L)$ is an inclusion of pairs of ordered
simplicial complexes 
with $L'=K'\cap L$ and $K=K'\cup L$. We show that $\widehat C^*(i)$
is surjective. Let $\phi'\in \widehat C^*(K',L')$. Define 
$\phi\in \widehat C^*(K,L)$
by $\phi(\sigma)=\phi'(\sigma)$ if $\sigma$ is a simplex in $K'$ and 
$\phi(\sigma)=0$ if $\sigma$ is a simplex in $L$. This defines $\phi$ coherently, 
since $\phi'(\sigma)=0$ when
 $\sigma\in K'\cap L=L'$, and exhaustively because $K=K'\cup L$.
 Clearly $\widehat C^*(i)(\phi)=\phi'$, hence $\widehat C^*(i)$ is sujective.
 For the injectivity notice that if $\phi\in \widehat C^*(K,L)$ is zero on each simplex
 of $K'$ then it is zero everywhere since it is zero on $L$ and  $K=K'\cup L$.
 This proves that   $\widehat C^*$ satisfies condition (B).
 
Suppose now that $C^*=\Apl$ is the functor of piecewise polynomial
forms and let $S_\bullet(X)$ be the 
simplicial set of singular simplices of a topological space $X$.
Recall from \cite[\S 10(c)]{FHT:RHT} that $\Apl\colon\Top\to
CDGA$ actually factors through the functor
$S_\bullet\colon\Top\to\sSets$  by the
way of another functor $\Apl\colon\sSets\to\CDGA$. We define
$$
\widehat C^*\colon\calK\to\CDGA\,,\,K\mapsto\Apl(K_\bullet).
$$
For any ordered simplicial complex $K$, the natural weak
equivalence $K_\bullet\quism S_\bullet(|K_\bullet|)$ induces a
quasi-isomorphism of CDGA
$$
C^*(|K|)=\Apl(|K|)\stackrel{\textrm{def}}=\Apl(S_\bullet(|K|))\cong
\Apl(S_\bullet(|K_\bullet|))\quism \Apl(K_\bullet)
\stackrel{\textrm {def}}=\widehat C^*(K).
$$
An element $\phi\in \widehat C^*(K,L)$ is determined by its values
in $\Apl^*$ on the non degenerated simplices of $K_\bullet$, hence on the genuine simplices of $K$.
The proof that $\widehat C^*$ satisfies
the strong excision isomorphism is analogous to the case $C^*=S^*(-;\Bk)$, the only difference is
that for the surjectivity one needs to make sure that the constructed  cochain $\phi$ commutes 
with the boundaries and degeneracies, which is straightforward.

\section{Group actions on DGmodules}
\label{S:groupactionsmodules}
 Let $G$ be a discrete group. Except when stated otherwise, we will suppose
that any action of a group is on the left and  that $G$
acts trivially on $\Bk$. 
When $G$ acts on two sets we will assume that it acts on its
product through the diagonal action. When $G$ acts on a set
equipped with additional algebraic structure, we will assume that this
action is such that each map defining this structure is
equivariant. In particular the action  on a
$\Bk$-module is linear, the action
on a tensor product is 
diagonal, $g{\cdot}(v\otimes w)=(g{\cdot}v)\otimes(g{\cdot}w)$, the action  on an
algebra $R$ is multiplicative, $g{\cdot}(r{\cdot}r')=(g{\cdot}r){\cdot}(g{\cdot}r')$. If $G$ acts on an algebra,
$R$, and on an $R$-module, $M$, then 
we assume that $g{\cdot}(r{\cdot}x)=(g{\cdot}r){\cdot}(g{\cdot}x)$ for $g\in G$, $r\in R$, and
$x\in M$. When $G$ acts on a graded object we assume that the action
preserves the degree and for the action  on a differential object $(M,d)$, that
the differential is equivariant, $d(g{\cdot}x)=g{\cdot}(dx)$. If $G$
acts on a $\Bk$-module $V$, then the \emph{dual action} of $G$ on
$\#V:=\homset(V,\Bk)$ is the  action defined by the formula, for
$\phi\in\#V$, $v\in V$, and $g\in G$,
$$\langle g{\cdot}\phi,v\rangle=\langle\phi,g^{-1}{\cdot}v\rangle.$$

To emphasize these assumptions, if $R$ is a $\Bk$-DGA on which $G$
is acting as above, we will say that $R$ is a $G$-$\Bk$-DGA. Also
if $M$ is an $R$-DGmodule on which $G$ is acting then we will say
that $M$ is a $G$-$R$-DGmodule. 

If $f\colon M\to N$ is a morphism of $G$-$R$-DGmodule, \ie an equivariant map of $R$-DGmodules,
then its mapping cone $Cf=N\oplus sM$ inherits a structure of $G$-$R$-DGmodule by
$g{\cdot}(y,sx)=(g{\cdot}y,s(g{\cdot}x))$ for $x\in M$ and $y\in N$.

Let $W$ be a topological space equipped with a left continuous
action of the group $G$ and recall from \refS{functorsC*} the functor $C^*$
which is either $S^*(-;\Bk)$ or $\Apl$.
As we noticed at the beginning of that section, an element $\omega\in C^*(W)$
is completely determined by its values $\langle\omega\,,\,\sigma\rangle$
(in $\Bk$ or in $\Apl^*$)
on simplices $\sigma\in S_\bullet(W)$.
We define an action of $G$ on $C^*(W)$ by
\begin{equation}
\label{E:actG}
\langle(g\cdot\omega)\,,\,\sigma\rangle\,=\,
\langle \omega\,,\,(g^{-1}\cdot\sigma) \rangle\quad
\textrm{ for }\omega\in C^*(W),\,\sigma\in S_\bullet(W),\,g\in G.
\end{equation} 
Since the action of $G$ on $S_\bullet(W)$ commutes 
with taking front face and back face, one checks that this induces
a structure of $G$-$\Bk$-DGA on $S^*(X;\Bk)$. It is also straightforward to check that it endows
$\Apl(W)$ with a structure of $G$-CDGA.

If $X\subset W$ is a subspace stable by $G$ (that is, $g\cdot X\subset X$ for $g\in G$)
then $C^*(X)$ is also a $G$-$\Bk$-(C)DGA and the restriction map 
$C^*(W)\to C^*(X)$, which is $G$-equivariant and of DGA, endows $C^*(X)$ with a structure of
$G$-$C^*(W)$-DGModule.

Suppose that $W$ is the realization of an ordered simplical complex, also denoted by $W$,
and that the action of $G$ respects that triangulation, \ie\ is simplicial and preserves the order of vertices.
Formulas \refN{E:actG} for simplices $\sigma$ of the simplicial complex $W$, defines
a structure of $G$-$\Bk$-(C)DGA on the chain complex $\widehat C^*(W)$
defined in \refS{functorsC*}. Also if $X$ is a subpolyhedron of $W$ stable by the action of $G$ then 
$\widehat C^*(X)$ becomes a $G$-$\widehat C^*(W)$-DGModule.

\section{Orientation twisted action and equivariant Lefschetz duality}
\label{S:equivD}
Let $W$ be a closed oriented connected manifold
of dimension $n$ on which $G$ acts continuously. We have an
induced action on the top homology group $H_n(W;\BZ)$. This
determines a $1$-dimensional representation over $\Bk$ through the homomorphism
\begin{equation}
\label{E:orrep}
\rho\colon G\to\BZ/2=\{+1,-1\}\subset\Bk,g\mapsto \rho(g)
\end{equation}
defined by the formula
$$
g{\cdot}[W]=\rho(g){\cdot}[W]
$$
where $[W]\in H_n(W;\BZ)$ is the orientation class. We call this the \emph{orientation representation}.\plfn{the referee
asked in remark 17 for more explanation or a reference. It seems to me that all of this is very clear, so I do not know what to add.}

\begin{defin}\label{D:orientation-twisted}
Let $G$ be a finite group acting continuously  on a closed connected oriented manifold $W$ and let $A$ be
a $G$-$R$-DGmodule.
The \emph{orientation-twisted action} of $G$ on $\#A$ is
defined by
$$
\langle g{\cdot}\phi,a\rangle
:=\rho(g)\langle\phi,g^{-1}{\cdot}a\rangle\quad\quad\textrm{for }
g\in G,\phi\in\#A,a\in A,
$$
where $\rho$ is the orientation representation \refN{E:orrep}.
\end{defin}
In particular if $X\subset W$ is a $G$-invariant subspace of $W$ we have an induced action
on  $C^*(W)$ and $C^*(X)$, and we can consider the orientation-twisted dual action on
$\#C^*(W)$ and $\#C^*(X)$.
The reason for introducing this twisted action is that it is the
correct one to make the Poincar\'e and Lefschetz
duality quasi-isomorphisms equivariant as we will see in \refT{equLefschetz}.
In order to prove that theorem we need the following proposition. 
Recall the functor $\widehat C^*$  from \refS{functorsC*}.
\begin{prop}\label{P:equivariantPoincareduality}
Let $K$ be an ordered simplicial complex whose realization 
 $W:=|K|$  is a
closed oriented connected manifold of dimension
$n$. Let $G$ be a finite group acting on the left on $K$ and
let $\Bk$ be a field such that $\mathrm{char}(\Bk)$
does not divide $|G|$.
 Then there exists a $\Bk$-DGmodule morphism
$$\epsilon_K\colon \widehat C^*(K)\to s^{-n}\Bk$$
such that $\epsilon_K^*\colon H^n(W)\iso H^n(s^{-n}\Bk)=\Bk$ is an isomorphism
and such that the $
C^*(W)$-DG-module morphism
$$
\Phi_K\colon\widehat C^*(K)\quism s^{-n}\#\widehat C^*(K),
$$
defined by $\Phi_K(\alpha)(\beta)=\epsilon_K(\alpha{\cdot}\beta)$, is
$G$-equivariant when $\widehat C^*(K)$ is equipped with the
standard dual action of $G$ and $\#\widehat C^*(K)$ is equipped
with the orientation-twisted $G$-action.
\end{prop}
\begin{proof}
Since $H^n(W)\cong\Bk$ there exists a chain map 
 $\epsilon'_K\colon\widehat C^*(K)\to s^{-n}\Bk$
 such that  $\epsilon'_K(\mu)=1$ for some cocycle $\mu$ representing the orientation class.
 Set
$\epsilon_K(\omega):=(1/|G|)\sum_{g\in
G}\rho(g){\cdot}\epsilon'_W(g{\cdot}\omega)$.
One computes that $\epsilon_K(\mu)=1$, so $\epsilon_K$ induces an isomorphism in homology.
One checks also that $\Phi_K$ is $G$-equivariant. The fact that $\Phi_K$ is a quasi-isomorphism is 
a consequence of Poincar\'e duality.
\end{proof}
We arrive to our cochain level equivariant Lefschetz duality theorem:
\begin{thm}\label{T:equLefschetz}
Let $W$ be an $n$-dimensional triangulated connected oriented
closed manifold.  Let $G$ be a finite group that acts simplicialy on
$W$. Let $f\colon X\hookrightarrow W$ be a
subpolyhedron stable by $G$.  Let $\Bk$ be a field such that
$\mathrm{char}(\Bk)$ does not divide $|G|$.
 Let  $C^*$ be the cochain algebra functor from \refS{functorsC*}.
Then there exists a
chain of weak equivalences of $G$-$C^*(W)$-DGmodules between
$C^*(W\smallsetminus X)$ and
$s^{-n}\left(\#C^*(W)\oplus_{\#C^*(f)}s\#C^*(X)\right)$, where
$\#C^*(W)$ and $\#C^*(X)$  are equipped
with the orientation-twisted dual $G$-action.
\end{thm}
\begin{proof}
As $W$ is triangulated, it is homeomorphic to the realization of an abstract simplicial complex.
Replace this simplicial complex by its second barycentric subdivision. The action of $G$ is still simplicial.
Moreover there is a natural structure of \emph{ordered} simplicial complex defined on the
barycentric subdivision as follows. Denote by $b(\sigma)$ the barycentre of a simplex $\sigma$ and
order the vertices of the barycentric subdivision by $b(\sigma)\leq b(\tau)$ if and only if
$\sigma\subset\tau$.\plfn{I added this paragraph to get easily \emph{ordered} simplicial complex}

For the rest of the proof, abusing  notation, we will denote  both
the manifold itself and the ordered simplicial
complex associated to this second barycentric subdivisions by the same
letter $W$, and similarly for other
subpolyhedra of $W$.\plfn{Referee in remark 22 asked to put this at the beginning of the proof but we need
first to introduce the 2nd barycentric subdivision since it is that one which will be denoted $W$}

 Let $T$ be the closure of the star of $X$ in the simplicial complex $W$.
 Since we have took the second subdivision,  $T$ is a regular neighborhood
of $X$, hence the inclusion  $i\colon X\hookrightarrow T$ 
is a $G$-equivariant homotopy equivalence.
  Since $G$ preserves $X$ it also preserves $T$ and
the boundary $\partial T$.
 Denote by $j\colon
T\hookrightarrow W$, $j_0\colon(T,\partial
T)\hookrightarrow(W,\overline{W\smallsetminus T})$, $i\colon
X\hookrightarrow T$ and $i'\colon \overline{W\smallsetminus T}
\hookrightarrow W $ the simplicial $G$-equivariant inclusions.

 Consider the functor $\widehat C^*$ defined
on ordered simplicial complexes in \refS{functorsC*}. By the strict
excision property we have an isomorphism
$$
\widehat C^*(j_0)\colon \widehat C^*(W,\overline{W\smallsetminus T})
\stackrel\cong\to  \widehat C^*(T,\partial T).
$$
Recall  the cochain
morphism
$$\epsilon_W\colon\widehat C^*(W)\quism s^{-n}\Bk$$
from \refP{equivariantPoincareduality} which induces the $G$-equivariant Poincar\'e duality
quasi-isomorphism
$$\Phi_W\colon\widehat C^*(W)\quism s^{-n}\#\widehat C^*(W)\,,\,\alpha\mapsto
\left(\Phi_W(\alpha)\colon\beta\mapsto\epsilon_W(\alpha\beta)\right).$$
Define $\epsilon_T$ as the composite
$$
\epsilon_T\colon \widehat C^*(T,\partial T)\stackrel{\widehat C^*(
j_0)}\cong\widehat C^*(W,\overline{W\smallsetminus
T})\stackrel{\iota}\to\widehat C^*(W)\stackrel{\epsilon_W}\to
s^{-n }\Bk.
$$
This cochain map serves to define a cochain morphism
$$\Phi_T\colon C^*(T,\partial T)\quism s^{-n}\#C^*(T)\,,\,\psi\mapsto
\left(\Phi_T(\psi)\colon\tau\mapsto\epsilon_T(\psi\tau)\right)
$$
which is a $G$-$C^*(W)$-DGmodule quasi-isomorphism.

Moreover, since $\epsilon_T$ is defined from $\epsilon_W$, the
following diagram of $G$-$C^*(W)$-DGmodules commutes
$$\xymatrix{
0\ar[r]&\widehat C^*(W,\overline{W\smallsetminus T})\ar[r]^-{\iota}
\ar[d]_-{\cong}^-{j^*_0}& \widehat
C^*(W)\ar[r]^-{i'^*}\ar[dd]_-{\simeq}^-{\Phi_W} & \widehat
C^*(\overline{W\smallsetminus T})\ar[r]&0\\
&\widehat C^*(T,\partial T)\ar[d]_-{\simeq}^-{\Phi_T}\\
&s^{-n}\#\widehat C^*(T)\ar[r]^-{s^{-n}\#j^*}&s^{-n}\#\widehat
C^*(W)\\
&s^{-n}\#\widehat
C^*(X)\ar[u]^-\simeq_-{s^{-n}\#i^*}\ar[r]^-{s^{-n}\#f^*}&s^{-n}\#\widehat
C^*(W).\ar@{=}[u] }$$
Therefore we have the following chain of
quasi-isomorphisms of $G$-$C^*(W)$-DGmodules \\
$\xymatrix@1{
&s^{-n}\left(\#\widehat C^*(W)\oplus_{\#f^*}s\#\widehat
C^*(X)\right)\ar[r]^-\simeq& s^{-n}\left(\#\widehat
C^*(W)\oplus_{\#j^*}s\#\widehat C^*(T)\right)&\ar[l]_-\simeq}$
\\
$\xymatrix@1{ &\ar[l]_-\simeq\widehat C^*(W)\oplus_{\iota}\widehat
C^*(W,\overline{W\smallsetminus T})\ar[r]^-\simeq_-{i'^*\oplus 0}&
\widehat C^*(\overline{W\smallsetminus T}).}$

The first and last terms of this zigzag of quasi-isomorphisms are
respectively 
 $s^{-n}\left(\# C^*(W)\oplus_{\#f^*}s\# C^*(X)\right)$
and
 $C^*(\overline{W\smallsetminus T})$. This proves the theorem.
\end{proof}
\section{Cubical diagrams and their total cofibres}
\label{S:cubical}
Let $E$ be a finite set and let
$\Gamma=(2^E)^{\operatorname{op}}$ be the category whose objects are subsets $\gamma$ of $E$ and a morphism
 $\gamma\to\gamma'$ is a reversed inclusion $\gamma\supset\gamma'$. The ``shape'' of this category
 is that of an $|E|$-dimensional cube with an initial object $E$ and a final object $\emptyset$.
\begin{defin}\label{D:cubicaldiag}
An $E$-cubical diagram in a category $\calC$ is a covariant functor
$$\barN\colon\Gamma\to\calC.$$
\end{defin}

For $\gamma\subset E$ we denote by $|\gamma|$ the cardinal of that
subset. If $e\in\gamma$ we set $\gamma\smallsetminus e:=\gamma\smallsetminus
\{e\}$.

Suppose a linear order  $\leq$ on $E$ has been given. For $e\in E$ and
$\gamma\in \Gamma$ we define the integer
$$\pos(e:\gamma):=|\{j\in\gamma:j\leq e\}|.$$
In other words if $\gamma=\{e_1,\ldots,e_l\}$ with $e_1<\cdots<e_l$ 
then $\pos(e_i:\gamma)=i$.

\begin{defin}\label{D:totalcofibre}
Let $R$ be a $\Bk$-DGA. Let $E$ be a finite set equipped with a
linear ordering. The \emph{total cofibre} of an
$E$-cubical diagram $\barN\colon\Gamma\to R\!-\!\DGMod$ of
$R$-DGmodules is the $R$-DGmodule
$$\TotCof(\barN):=\left(\oplus_{\gamma\in\Gamma}\,y_\gamma{\cdot}\barN(\gamma),D\right)
$$
where, for $\gamma\in\Gamma$, $x\in\barN(\gamma)$, $r\in R$,
\begin{itemize}
\item $y_\gamma$ is a variable of degree $-|\gamma|$;
\item $\deg(y_\gamma{\cdot}x)=-|\gamma|+\deg(x)$;
\item $r{\cdot}(y_\gamma{\cdot}x)=(-1)^{|\gamma|\deg(r)}y_\gamma{\cdot}(r{\cdot}x)$;
\item $D(y_\gamma{\cdot}x)=(-1)^{|\gamma|}y_\gamma{\cdot}(d(x))+
\sum_{e\in\gamma}y_{\gamma\smallsetminus
e}(-1)^{\pos(e:\gamma)}\barN(\gamma\to\gamma\smallsetminus e)(x).$
\end{itemize}
\end{defin}
The notion of  a total cofibre of a cubical diagram was first introduced by Goodwillie
in \cite{Goodwillie:Calculus2}. Actually it is a special case of the following more general definition,
see \cite{Hut:TC}. Let $\Gamma$ be a poset and $\Gamma'\subset \Gamma$ be a subposet and let
$X\colon \Gamma\to\calC$ be a covariant functor in some Quillen model
category. The total cofibre of $X$ is 
defined as the homotopy cofibre of the map
\[\hocolim\limits_{\gamma'\in\Gamma'}X(\gamma')\to\hocolim\limits_{\gamma\in\Gamma}X(\gamma).
\]
In our case $\calC$ is the category $R$-DGMod, $\Gamma=(2^E)^{\operatorname{op}}$ and $\Gamma'$ 
is $\Gamma$ without its final object $\emptyset$. 

Notice that the definition of the total cofibre depends on the
choice of a linear ordering on $E$ but is easy
 to check that two
such linear ordering give isomorphic total cofibres. (Hint: Use \refL{sgngpos}.)

We introduce the notion of an iterated mapping cone of a bounded
chain complex in $R$-DGMod, which extends the usual mapping
cone of a chain map. Let
$$M_*:=\{ M_r\stackrel{f_r}\to M_{r-1}\stackrel{f_{r-1}}\to\cdots
\stackrel{f_2}\to M_{1}\stackrel{f_{1}}\to M_0\}
$$
where $(M_i,d_i)$ and $f_i$ are objects and morphisms in $R$-DGMod,
for some DGA $R$, such that $f_if_{i+1}=0$.
The \emph{iterated mapping cone} of
$M_*$ is defined as
$$C(M_*):=\left(\oplus_{i=0}^rs^iM_i,D\right)$$
with $D(s^ix)=(-1)^is^i(d_i x)+s^{i-1}f_i(x)$, for $x\in M_i$. It
is straightforward to check that $D^2=0$.
When $r=1$ this
is the usual mapping cone of the map $f_1\colon M_1\to M_0$.

If $\barN\colon\Gamma\to\DGMod$ is an $E$-cubical diagram of
DGmodules with $|E|=r$ then we can define a bounded complex of
DGmodules
\begin{equation}\label{E:N*}
N_*:=\{
N_r\stackrel{f_r}\to N_{r-1}\stackrel{f_{r-1}}\to\cdots
\stackrel{f_2}\to N_{1}\stackrel{f_{1}}\to N_0\}
\end{equation}
with $$N_i:=\underset{\gamma\in\Gamma,|\gamma|=i}\oplus\,\barN(\gamma)$$
and, for $x\in \barN(\gamma)\subset N_i$,
$$f_i(x):=\sum_{e\in\gamma}(-1)^{\pos(e:\gamma)}\barN(\gamma\to\gamma\smallsetminus
e)(x).$$ Then it is straightforward to check that $N_*$ is a
complex of differential modules and that the total cofibre of the cube
$\barN$ coincides with the iterated mapping cone of $N_*$:
$$
\TotCof(\barN)\cong C(N_*).
$$
\section{$G$-action on a cubical diagram of $R$-DGmodules}
\label{S:Gcubical}
Let $E$ be a finite set equipped with an action of $G$. This
induces a $G$-action on the poset $\Gamma=2^E$ that preserves the
order (induced by reverse inclusions). Let $\barN\colon\Gamma\to
R\!-\!\DGMod$ be an $E$-cubical diagram $R$-DGmodules, where $R$ is a
$G$-$\Bk$-DGA.

By a \emph{$G$-action on $\barN$} we mean the data of $\Bk$-linear
morphisms
$$
\barN(g,\gamma)\colon\barN(\gamma)\to\barN(g{\cdot}\gamma),
$$
for each $g\in G$ and $\gamma\in\Gamma$ such that
\begin{description}
\item[$G$-naturality] the following diagrams commute
$$\xymatrix{\barN(\gamma)\ar[d]_-{\barN(g,\gamma)}\ar[r]^-{\barN(\gamma\to\gamma')}&
\barN(\gamma')\ar[d]^-{\barN(g,\gamma')}\\
\barN(g{\cdot}\gamma)\quad\ar[r]_-{\barN(g{\cdot}\gamma\to
g{\cdot}\gamma')}&\quad\barN(g{\cdot}\gamma')\quad;}
$$
\item[associativity]
$\barN(g',g{\cdot}\gamma)\barN(g,\gamma)=\barN(g'g,\gamma)$;
\item[unit]$\barN(1,\gamma)=\id$ where $1\in G$ is the identity;
\item[$G$-$R$-module] for $x\in\barN(\gamma)$ and $r\in R$,
$\barN(g,\gamma)(r{\cdot}x)=(g{\cdot}r){\cdot}\left(\barN(g,\gamma)(x)\right).$
\end{description}
\vspace{3mm}

For $g\in G$, $\gamma\in\Gamma$ and $x\in\barN(\gamma)$ we simply
write $g{\cdot}x$ for $\barN(g,\gamma)(x)$ when there is no possible
confusion. Then the associativity and unit axioms are the usual
axioms $(g'{\cdot}g){\cdot}x=g'{\cdot}(g{\cdot}x)$ and $1{\cdot}x=x$, and the $G$-$R$-module
axioms means that $g{\cdot}(r{\cdot}x)=(g{\cdot}r){\cdot}(g{\cdot}x)$. In particular the maps
$\barN(g,\gamma)$ are \emph{not} maps of $R$-DGmodule.

 Notice that if $G$
acts on the $E$-cube $\barN$ then in particular $G$ acts on the
$R$-DGmodule $\barN(\emptyset)$.

Suppose given such an $E$-cubical diagram $\barN$ of $R$-DGModules
equipped with a $G$-action as defined above. Fix a linear ordering
on $E$. Our goal is to define a $G$-action on the total
cofibre $\TotCof(\barN)$ making it a  $G$-$R$-DGmodule. Notice that the
``obvious'' action $g{\cdot}(y_\gamma{\cdot}x)=y_{g{\cdot}\gamma}{\cdot}\barN(g,\gamma)(x)$ 
for $\gamma\in\Gamma$, $g\in G$ and $x\in \barN(\gamma)$
is not the correct one because it does not make the differential
equivariant.

\begin{defin}
\label{D:sgn}Let $\phi\colon L\iso L'$ be an bijection between two finite 
linearly ordered set of cardinality $r\geq0$, non-necessarily order-preserving. We define its \emph{signature},
 $\sgn(\phi)$, as the signature  of the permutation in $\Sigma_r$
obtained as the composite
$$\{1,\cdots,r\}\stackrel{\psi}\to L\stackrel{\phi} \to L' \stackrel{\psi'}\to \{1,\cdots,r\}$$
where $\psi$ and $\psi'$ are the unique order-preserving bijections.\\
If $E$ is a linearly ordered finite set with an action of a finite group $G$
then for all subset $\gamma\subset E$
and all $g\in G$ the restriction to $\gamma$ gives 
 a bijection $g|\gamma\colon \gamma\iso g{\cdot}\gamma,e\mapsto g{\cdot}e$,
and we denote its signature by $\sgn(g:\gamma)$, where $\gamma$ and $g{\cdot}\gamma$ are equipped with the linear order
induced by $E$.
\end{defin}

\begin{lemma}\label{L:sgngpos}
Let $\gamma\in\Gamma=(2^E)^{\operatorname{op}}$, let $g\in G$ and let $e\in\gamma$. Then
$$\sgn(g:\gamma){\cdot}\sgn(g:\gamma\smallsetminus e)=(-1)^{\pos(e:\gamma)}{\cdot}(-1)^{\pos(g{\cdot}e:g{\cdot}\gamma)}.$$
\end{lemma}
\begin{proof}
Straightforward.
\end{proof}

Define an action of $G$ on the total cofibre, $\TotCof(\barN)$, of \refD{totalcofibre}
by
\begin{equation}
\label{E:gy}
g{\cdot}y_\gamma:=\sgn(g:\gamma){\cdot}y_{g{\cdot}\gamma}
\end{equation}
inducing
$$g{\cdot}(y_\gamma{\cdot}x):=(g{\cdot}y_\gamma){\cdot}(g{\cdot}x)=\sgn(g:\gamma){\cdot}y_{g{\cdot}\gamma}{\cdot}\barN(g,\gamma)(x).$$

\begin{prop}\label{P:actionTotCof}
The action defined above induces a $G$-$R$-DGmodule structure on
$\TotCof(\barN)$ such that the inclusion
$\barN(\emptyset)\hookrightarrow\TotCof(\barN)$ is
$G$-equivariant.
\end{prop}
\begin{proof}
Use \refL{sgngpos} to prove that the differential is equivariant.
\end{proof}

\section{Equivariant Lefschetz theorem for a system of subpolyhedra}
\label{S:equivLefschetzSystPolyh}

Let $W$ be a triangulated space. Let $E$ be a finite set and let
$$\{j_e\colon X_e\hookrightarrow W\}_{e\in E}
$$
be a collection of subpolyhedra  indexed by $e\in E$.

Recall from \refS{cubical} the category $\Gamma=(2^E)^{\operatorname{op}}$. For
$\emptyset\not=\gamma\in\Gamma$ set
$$X_\gamma:=\cap_{e\in\gamma}X_e$$
and set 
$$X_\emptyset:=W.$$
This defines a cubical diagram $X_\bullet\colon\Gamma\to\Top\,,\,\gamma\mapsto X_\gamma$,
with the reversed inclusion $\gamma\supset\gamma'$ sent to the inclusion $X_\gamma\hookrightarrow X_{\gamma'}$.

Each $C^*(X_\gamma)$ is a \emph{right} $C^*(W)$-DGmodule,
therefore its dual $\#C^*(X_\gamma)$ is a left $C^*(W)$-DGmodule.
Moreover if $\gamma\supset \gamma'$, the inclusion map
$X_\gamma\hookrightarrow X_{\gamma'}$ induces a morphism
$$
\#C^*(X_\gamma)\to\#C^*(X_{\gamma'}).
$$
In other words we have an $E$-cubical diagram of $C^*(W)$-DGmodules
$$
\#C^*(X_\bullet)\colon\Gamma\to C^*(W)\!-\!\DGMod\,,\,\gamma\mapsto
\#C^*(X_\gamma).
$$
Fix a linear ordering on $E$ and consider the total
cofibre of $\#C^*(X_\bullet)$. The following is a folklore fact:
\begin{prop}\label{P:acyclictotcof}
With the setting above, if $W=\cup_{e\in E}X_e$ then the total
cofibre $\TotCof(\#C^*(X_\bullet))$ is acyclic.
\end{prop}
\begin{proof}
When\plfn{I was lazy in the previous version but since the referee ask for it here is the  long proof of
this quite of folklore fact}  $|E|\leq1$ the proposition is trivial and for $|E|=2$ it is exactly Mayer-Vietoris theorem. We prove the general case by an induction
on the cardinality of $E$. Suppose that the proposition has been proved for $|E|\leq k$ and let
$E=E_0\cup\{a\}$ with $|E_0|=k$. Set $\Gamma_0=(2^{E_0})^{\operatorname{op}}$,
 $W_0=\cup_{e\in E_0}X_e$, $\Gamma=(2^{E})^{\operatorname{op}}$, and
  $W=\cup_{e\in E}X_e$.

Consider the three systems of subpolyhedra
$\{X_e\}_{e\in E_0}$, $\{X_e\}_{e\in E}$,
and $\{X_a\cap X_e\}_{e\in E_0}$.
The corresponding  total cofibres of the  associated diagrams,
$\TotCof(\{\#C^*(X_\gamma)\}_{\gamma\in\Gamma_0})$,
$\TotCof(\{\#C^*(X_\gamma)\}_{\gamma\in\Gamma})$, and
$\TotCof(\{\#C^*(X_a\cap X_\gamma)\}_{\gamma\in\Gamma_0})$,
are obtained as the iterated mapping cones of the bounded chain complexes $A'_*$, $A_*$, and $A''_*$
defined as follows: For $r\geq 1$, we have
\begin{eqnarray*}
A'_r&=&\oplus_{\gamma\in\Gamma_0,|\gamma|=r}\,\#C^*(X_\gamma),\\
A_r&=&\oplus_{\gamma\in\Gamma,|\gamma|=r}\,\#C^*(X_\gamma),\\
A''_r&=&\oplus_{\gamma\in\Gamma_0,|\gamma|=r}\,\#C^*(X_a\cap X_\gamma),
\end{eqnarray*}
and
\begin{eqnarray*}
A'_0&=&\#C^*(W_0),\\
A_0&=&\#C^*(W),\\
A''_0&=&\#C^*(X_a\cap W_0).
\end{eqnarray*}
For $r\geq 2$ we have obvious short exact sequences
\begin{equation}\label{E:exAr}
0\to A'_r\to A_r\to A''_{r-1}\to 0
\end{equation}
as well as a short exact sequence
\begin{equation}\label{E:exA1}
0\to A'_1\to A_1\to \#C^*(X_a)\to 0.
\end{equation}
Mayer-Vietoris theorem for $W=W_0\cup X_a$ implies that the commutative square
\[
\xymatrix{A''_0\ar[r]\ar[d]_{q_0}&\#C^*(X_a)\ar[d]^q\\
A'_0\ar[r]_i&A_0
}
\]
induces a quasi-isomorphism between the mapping cones of the two horizontal arrows of this square 
\begin{equation}\label{E:MVqi}
q\oplus sq_0\colon\#C^*(X_a)\oplus sA''_0\quism A_0\oplus_isA'_0.
\end{equation}
The bulk of the proof is the study of the following commutative diagram of DGmodules
\begin{equation}\label{E:diagMV}
\xymatrix{
0\ar[r]\ar[d]&A'_k\ar[r]\ar[d]&\cdots\ar[r]&A'_2\ar[r]\ar[d]&A'_1\ar[r]\ar[d]&A'_0\ar[d]^i\\
A_{k+1}\ar[r]\ar[d]^{=}&A_k\ar[r]\ar[d]&\cdots\ar[r]&A_2\ar[r]\ar[d]&A_1\ar[r]\ar[d]&A_0\ar@{^(->}[d]\\
A''_{k}\ar[r]&A''_{k-1}\ar[r]&\cdots\ar[r]&A''_1\ar[r]^-p&\#C^*(X_a)\ar[r]^q&A_0\oplus_isA'_0.
}
\end{equation}
where portions of  the horizontal lines are  the chain complexes $A'_*$, $A_*$, and $A''_*$,
$p$ is the composite of the map $A''_1\to A''_0=\#C^*(X_a\cap W_0)$
with the map $\#C^*(X_a\cap W_0)\to \#C^*(X_a)$ induced by the inclusion
$X_a \cap W_0\hookrightarrow X_a$, and the vertical arrows are short exact sequences
\refN{E:exAr} and \refN{E:exA1} except for the rightmost which is the obvious sequence of the mapping cone of $i\colon A'_0\to A_0$.

For the sake of the proof we say that a bounded
chain complex of DGmodules is \emph{quasi-exact} if its iterated mapping cone
is acyclic. We need to prove that the middle horizontal line of Diagram \refN{E:diagMV}
is quasi-exact.

Each vertical sequence in  \refN{E:diagMV} is quasi-exact because it is either a short exact sequence or
it is the sequence of a mapping cone. The top horizontal line $A'_*$ is quasi-exact by
the  induction hypothesis applied to the system
$\{X_e\}_{e\in E_0}$. We claim that the bottom horizontal line is also quasi-exact.
Indeed by induction hypothesis $A''_*$ is quasi-exact. Therefore the iterated mapping cone of the truncated
bounded chain complex $\{A''_k\to\dots\to A''_1\}$ is quasi-isomorphic to $A''_0$.
By \refN{E:MVqi} we deduce a quasi-isomorphism
\[
\#C^*(X_a)\oplus sC(\{A''_k\to\dots\to A''_1\})\,\quism\,A_0\oplus_isA'_0
\]
which implies the claim.

We can take the iterated mapping cone
 of each vertical sequences in Diagram \refN{E:diagMV} and then take the
 iterated mapping cone of the horizontal chain complexes obtained from these iterated mapping cones.
We get an acyclic DG-module because each term the horizontal complex of iterated mapping cones is acyclic
since the vertical sequences are quasi-exact.
 Working in the opposite order we can first take the iterated mapping cones of each of the three horizontal lines of
 Diagram \refN{E:diagMV}, then take the iterated mapping of the resulting  chain complex  of these three iterated
 mapping cones.
 This iterated mapping cone is also acyclic
 because the result is independent on the order between the horizontal and vertical directions.
 Moreover we have proved that the iterated mapping cone of the top and the bottom horizontal lines are acyclic.
 Therefore the iterated mapping cone of
 the middle horizontal line is also acyclic.
\end{proof}

 Suppose that  $W$ is an
oriented connected closed manifold of dimension $n$.
Let $G$ be a finite  group acting continuously on  $W$.
Suppose that $G$ also acts on the set $E$ in such a way that
$g{\cdot}(X_e)=X_{g{\cdot}e}$ for $g\in G$ and $e\in E$. This induces a
$G$-action on the $E$-cubical diagram $\barN:=\#C^*(X_\bullet)$ as
follows. Recall the orientation representation $\rho$ of  \refN{E:orrep}.
 For $g\in G$ and $\gamma\in \Gamma$ define a
morphism
\begin{equation}
\label{E:NtwG}
\barN(g,\gamma)\colon\#C^*(X_\gamma)\to\#C^*(X_{g{\cdot}\gamma})
\end{equation}
as the morphism induced through $\#C^*$ by the continuous map
$g\colon X_\gamma\to X_{g{\cdot}\gamma}$ multiplied by the sign
$\rho(g)$. This $G$-action on the cube $\#C^*(X_\bullet)$ is
called the \emph{orientation-twisted} action. It is
straightforward to check that it defines an action of $E$-cube of
$C^*(W)$-DGModules.

\begin{thm}
\label{T:EquivLefschSystem} Let $W$ be a triangulated oriented
connected closed manifold of dimension $n$. Let $E$ be a finite
set and let
$$\{j_e\colon X_e\hookrightarrow W\}_{e\in E}
$$
be a collection of subpolyhedra  indexed by $e\in E$. Let $G$ be a
finite group acting continuously on the manifold $W$. Suppose that
$G$ also acts on the set $E$ in such a way that $g{\cdot}(X_e)=X_{g{\cdot}e}$
for $g\in G$ and $e\in E$. Let $\Bk$ be a field and assume that
$\mathrm{char}(\Bk)$ does not divide $|G|$. Let $C^*$ be
the algebra cochain  functor of \refS{functorsC*}.

Consider the $E$-cubical diagram of $C^*(W)$-DGmodules
$$\#C^*(X_\bullet)\colon\Gamma=(2^E)^{\textrm{\emph{op}}}\to C^*(W)\!-\!\DGMod$$
defined above equipped with the orientation-twisted $G$-action
\refN{E:NtwG} and
consider the induced action on its total cofibre as in
\refP{actionTotCof}.

Then there is a chain of weak equivalences of
$G$-$C^*(W)$-DGmodules between
\begin{itemize}
\item[(i)] $C^*(W\smallsetminus\cup_{e\in E}X_e)$ and
\item[(ii)] $s^{-n}\TotCof(\#C^*(X_\bullet))$.
\end{itemize}
\end{thm}

\begin{proof}
We first construct a short  sequence of $E$-cubical diagrams of
$G$-$C^*(W)$-DGmodules
\begin{equation} \label{E:N'NN"}
\xymatrix{ 0\ar[r]&\barN'\ar[r]^-\mu&\barN\ar[r]&\barN''\ar[r]&0.}
\end{equation}
 Set $\barN(\gamma):=\#C^*(X_\gamma)$ with in particular
$\barN(\emptyset)=\#C^*(W)$. Define $\barN'$ exactly as $\barN$
except that $\barN'(\emptyset)=\#C^*(\cup_{e\in E}X_e)$. The
inclusion $f\colon \cup_{e\in E}X_e\hookrightarrow W$ induces a map
$\mu(\emptyset)\colon \barN'(\emptyset)\to \barN(\emptyset)$
which combined with the identity maps, $\mu(\gamma)=\id$ for
$\gamma\not=\emptyset$, gives a morphism of $E$-cubical diagram
$$
\mu=\{\mu(\gamma)\}_{\gamma\in\Gamma}\colon\barN'\to\barN.
$$
Let $\barN''$ be the objectwise mapping cone of $\mu$, that is
$$
\barN'':=\{\barN(\gamma)\oplus_{\mu(\gamma)}
s\barN'(\gamma)\}_{\gamma\in\Gamma}.
$$
For $\gamma\not=\emptyset$, $\barN''(\gamma)$ is acyclic because it
is the mapping cone of the identity map. Therefore the total
cofibre of $\barN''$ is quasi-isomorphic to
$$
\barN''(\emptyset)=\#C^*(W)\oplus_{\#f^*}s\#C^*(\cup_{e\in E}X_e).
$$
By \refT{equLefschetz} we deduce that $C^*(W\smallsetminus \cup_{e\in
E}X_e)$ is weakly equivalent to $s^{-n}\TotCof(\barN'')$, and this
weak equivalence is $G$-equivariant by \refP{actionTotCof}.

On the other hand, since $\barN''$ is the mapping cone of
$\barN'\to\barN$, the short complex of cubes of differential
modules \refequ{E:N'NN"} induces  a long exact sequence between the
homologies of these cubes. Therefore it also induces a long exact
sequence between the homology of their total cofibres. Moreover by
\refP{acyclictotcof} the total cofibre of $\barN'$ is acyclic,
hence we deduce that $s^{-n}\TotCof(\barN'')$ and
$s^{-n}\TotCof(\barN)$ are weakly equivalent. We have shown above
that the first one is weakly equivalent to (i), and the second one
is (ii).
\end{proof}

\section{Models for configuration spaces}\label{S:modelconf}
We come now to the proof of our  main result. 

Fix an integer $k\geq1$ and let $M$ be  a closed oriented triangulated manifold   of dimension
$m$. Consider the  action of the symmetric
group $\Sigma_k$ on $M^k $ and on the configuration space $F(M,k)$
by permutation of the factors.
Let $A=(A,d)$ be a connected Poincar\'e duality CDGA
weakly equivalent to $\Apl(M)$ and suppose given quasi-isomorphisms of CDGA
\[
\xymatrix{\Apl(M)&\ar[l]_-\simeq R\ar[r]^-\simeq&A}
\]
(recall that this exists by the main result of \cite{LS:PDCDGA-x2}.)
The inclusion $F(M,k)\hookrightarrow M^k$ induces a structure of $\Sigma_k$-$\Apl(M^k)$-DGmodules, hence of $\Sigma_k$-$R^{\otimes k}$-DGmodules,
on $\Apl(F(M,k))$. There is also an obvious structure of $\Sigma_k$-$R^{\otimes k}$-DGmodules on the CDGA $F(A,k)$ of \refD{FAk}.
\begin{thm}\label{T:main} With the above setting, 
there is a weak equivalence of $\Sigma_k$-$R^{\otimes k}$-DGmodules between
 $\Apl(F(M,k))$ and $F(A,k)$.
\end{thm}
The rest of the section is devoted to the proof of that theorem.

The triangulation of $M$ induces a triangulation on the $k$-fold product $W:=M^k$ compatible with the action of
 $\Sigma_k$ (Hint: Take the triangulation induced by the prismatic decomposition of $\Delta^{p_1}\times\cdots\times\Delta^{p_k}$
 after fixing a linear order on the vertices of $M$.)
  
Switching two factors of $M^k$ induces a self-map of degree $(-1)^m$.
Therefore the orientation representation associated to the action of $\Sigma_k$ on $M^k$
is given by, for $\sigma\in\Sigma_k$,
$$\rho(\sigma)=(\sgn(\sigma))^m.$$

Let $E$ be the set
$$E:=\{(i,j):1\leq i<j\leq k\}$$
linearly ordered by the left lexicographic order. This can be considered as the set of (non-oriented) edges
on the set of vertices $\underline{k}=\{1,\cdots,k\}$. Then the
 objects of $\Gamma=(2^E)^{\operatorname{op}}$  can be interpreted as simple graphs (no
loops, no double edges, no orientations on the edges) with vertices in $\underline{k}$.

Suppose given a graph $\gamma\in\Gamma$ and a permutation $\sigma\in\Sigma_k$ . The lexicographic order on  $E$ induces a linear order on $\gamma\subset E$
and $\sigma$ induces a bijection $\sigma\colon \gamma\iso\sigma{\cdot}\gamma$. We can consider its 
signature $\sgn(\sigma:\gamma)$ as in \refD{sgn}, not to be confused with $\sgn(\sigma)$.
We denote by $\pi_0(\gamma)$ the set of connected components of the graph
$\gamma$. In other words $\pi_0(\gamma)=\{\gamma(1),\cdots,\gamma(r)\}$
is a partition of $\underline k$ such that two vertices $i$ and $j$ belongs to the same set
$\gamma(u)$ if and only if they are connected by a path in $\gamma$. We define a linear order on  $\pi_0(\gamma)$ by
$$\gamma(u)\leq \gamma(v)\quad\Longleftrightarrow \quad
\min \gamma(u)\leq\min \gamma(v).$$
We denote by $\|\gamma\|:=|\pi_0(\gamma)|$ the number of path components of $\gamma$, not
to be confused with the  number $|\gamma|$ of edges in $\gamma$. A permutation $\sigma\in \Sigma_k$
also induces a bijection $\sigma\colon\pi_0(\gamma)\iso\pi_0(\sigma{\cdot}\gamma)$ of ordered sets and
we denote its signature by $\sgn(\sigma:\pi_0(\gamma))$.

For $(i,j)\in E$ set
$$X_{(i,j)}:=\{(x_1,\cdots,x_k)\in M^k|x_i=x_j\}\subset W=M^k$$
which defines a system of subpolyhedra $\{X_e\subset M^k\}_{e\in E}$. The associated cubical diagram as
in \refS{equivLefschetzSystPolyh}
is the functor
$$
\barDelta\colon\Gamma\to\Top,\gamma\to\barDelta(\gamma)=\{x_1,\cdots,x_k)\in
M^k|x_i=x_j\textrm{ if }(i,j)\in\gamma\}=\cap_{e\in\gamma}X_e.
$$
where the morphisms $\barDelta(\gamma\to\gamma')$ are the obvious
inclusions $\barDelta(\gamma)\hookrightarrow \barDelta(\gamma')$.

We define another functor, $\barM$, naturally homeomorphic to
$\barDelta$. For $\gamma\in\Gamma$, set  $\pi_0(\gamma)=\{\gamma(1),\cdots,\gamma(r)\}$
with $\gamma(u)<\gamma(v)$ if $1\leq u<v\leq r=\|\gamma\|$.  Set
$\barM(\gamma)=M^{\times\|\gamma\|}$, the $r$-fold product of $M$.
We have a homeomorphism
$$h(\gamma)\colon\barM(\gamma)\to\barDelta(\gamma),(y_1,\cdots,y_r)\to(x_1,\cdots,x_k)
$$
defined by $x_i=y_u$ if $i\in\gamma(u)$.
It is easy to make $\barM$ into a functor such that
the homeomorphism $h\colon\barM\cong\barDelta$ is natural.

By an  \emph{iterated diagonal} we mean a diagonal map
$ M\to M^p,x\mapsto(x,\cdots,x)$, for $p\geq 0$. Each map
$$
\barM(\gamma\to\gamma')\colon M^{\times\|\gamma\|}\to
M^{\times\|\gamma'\|}
$$
is the composite of a product of $\|\gamma\|$ iterated diagonals
followed by a permutation of the $\|\gamma'\|$ factors. These are uniquely determined by $h(\gamma)$.

Recall the Kunneth quasi-isomorphism
$$
\Apl(M)^{\otimes r}\quism \Apl(M^{\times r})\,,\,
a_1\otimes\cdots\otimes a_r\mapsto
\pr_1^*(a_1){\cdot}\ldots{\cdot}\pr_r^*(a_r)
$$
Through this Kunneth quasi-isomorphism,
diagonal maps correspond exactly to the mutiplication, and
permutations of the factors of $M^{\times r}$ correspond to
permutations (with a Koszul sign) of $\Apl(M)^{\otimes r}$.
Therefore for each $\gamma\leq\gamma'\in\Gamma$ there exists a
morphism, obtained as a graded signed permutation followed by iterated
multiplications,
$$
\mu(\gamma\to\gamma')\colon \Apl(M)^{\otimes\|\gamma'\|}\to
\Apl(M)^{\otimes\|\gamma\|}
$$
making the following diagram commutes
$$
\xymatrix{
\Apl(M)^{\otimes\|\gamma\|}\ar[d]_-\simeq^-{\textrm{Kunneth}}\quad\quad&
\ar[l]_-{\mu(\gamma\to\gamma')}\quad\quad
\Apl(M)^{\otimes\|\gamma'\|}\ar[d]_-\simeq^-{\textrm{Kunneth}}
\\
\Apl(\barM(\gamma))\quad\quad&
\ar[l]^-{\Apl(\barM(\gamma\to\gamma'))}
\quad\quad\Apl(\barM(\gamma')). }
$$
In other words we have build a contravariant functor
$$
\Gamma\to\CDGA\,,\,\gamma\mapsto\Apl(M)^{\otimes\|\gamma\|}
$$
with the morphisms $\mu(\gamma\to\gamma')$ obtained as a
permutation followed by iterated multiplication and it is
 naturally quasi-isomorphic to the contravariant
functor $\Apl(\barM)\colon\Gamma\to\CDGA$.
Using the quasi-isomorphisms
$$
\xymatrix{\Apl(M)&R \ar[l]_-\simeq \ar[r]^-\simeq& A}
$$
and considering similar composite of permutations and iterated
multiplications on iterated tensor products of $A$ or $R$, we can
build a contravariant functor
$$
\barA\colon\Gamma\to\CDGA\,,\,\gamma\mapsto\barA(\gamma)=A^{\otimes\|\gamma\|}
$$
naturally weakly equivalent to $\Apl(\barM)$.

The duals of the above diagrams give $E$-cubical
diagrams of $R^{\otimes k}$-DGmodules. We equip them with 
the orientation-twisted dual of the action of 
$\Sigma_k$ on the $k$-fold tensor product by permutation of the factors with a Koszul sign. 
By \refT{EquivLefschSystem} the
$(km)$-th suspension of the total cofibre of the cubical diagram $\#\Apl(\barDelta)$
is equivariantly weakly equivalent to 
$$\Apl(M^k\smallsetminus\cup_{(i,j)\in E}X_{(i,j)})=\Apl(F(M,k)),$$
 as $\Sigma_k$-$R^{\otimes k}$-DGmodules. Therefore $\Apl(F(M,k))$ is quasi-isomorphic as
$\Sigma_k$-$R^{\otimes k}$-DGmodules to
\begin{equation}
\label{E:TotCofA}
s^{-mk}\TotCof(\#\barA)=
s^{-mk}\left(\oplus_{\gamma\in\Gamma}\,y_\gamma{\cdot}\#(A^{\otimes\|\gamma\|})\right).
\end{equation}

To finish the proof of  \refT{main} we will build a
quasi-isomorphism of $\Sigma_k$-$R^{\otimes k}$-DGmodules
$$
\Phi\colon s^{-mk}\TotCof(\#\barA)\to F(A,k).
$$
This is the content of the following series of lemma. 

First we introduce some notation and terminology.\\
We say that a graph $\gamma\in\Gamma$ is \emph{redundant} if there exists an edge $e\in\gamma$
such that $\|\gamma\smallsetminus e\|=\|\gamma\|$. \\
For an edge $e=(i,j)\in E$ with $1\leq i<j\leq k$ we set
$\Delta(e):=\pi^*_{ij}(\Delta)\in A^{\otimes k}$ where
$\Delta\in A\otimes A$ is the diagonal class of \refE{Delta}.\\
For a graph $\gamma\in\Gamma$, we set 
$$g_\gamma:=\prod_{e\in\gamma}g_e\in F(A,k)$$
where the product is taken in the lexicographic order of $\gamma\subset E$ and
$g_e=g_{ij}$ is the $m-1$-dimensional generator of $F(A,k)$. Notice that if $\gamma$ is a redundant graph
then $g_\gamma=0$ because of the Arnold relations and $(g_{ij})^2=0$.\\
For the Poincar\'e duality algebra $A$ with orientation form $\omega\in\#A^m$ we denote
by $[A]\in A^m$ its fundamental class characterized by $\omega([A])=1$. 
We have a unique degree $-\!m$ isomorphism of $A$-module $\theta\colon A\iso\#A$ characterized
by $\theta(1)=\omega$. For $r\geq 1$ we denote by $\epsilon_r\in \#(A^{\otimes r})^{rm}$ the
linear form characterized by 
$\epsilon_r([A]\otimes\cdots\otimes[A])=1$.\\
The multiplication of the algebra $A$ is denoted by $\operatorname{mult}\colon A\otimes A\to A$.

Our first three lemmas aim to give an explicit formula for the differential $D$ in the total cofibre
\refN{E:TotCofA}.
\begin{lemma}
\label{L:multth1}
$(\#\operatorname{mult})(\theta(1))=\pm(\theta\otimes\theta)(\Delta) \in \#(A\otimes A)$.
\end{lemma}
\begin{proof}
Evaluate both sides on the basis $\{a_\lambda\otimes a_\mu^*\}$.
\end{proof}

\begin{lemma}
\label{L:Aeps}
Let $\gamma\in\Gamma$, $e\in\gamma$ and set $r=\|\gamma\|$. Then
there exist signs  $\nu(\gamma,e)\in\{-1,+1\}$ such that
\begin{equation*}
(\#\barA(\gamma\to\gamma\smallsetminus e))(\epsilon_r)=
\begin{cases} \nu(\gamma,e)\Delta(e){\cdot}\epsilon_{r+1} & \text{if $\|\gamma\smallsetminus e\|>\|\gamma\|$,}\\
\epsilon_{r}&\text{otherwise.}
\end{cases}
\end{equation*}
\end{lemma}
\begin{proof}
In the first case this map is the dual of a signed permutation of $A^{\otimes r+1}$
followed by a multiplication of two adjacent factors. An argument analogous to that of \refL{multth1}
by evaluation on a basis of $A^{\otimes r+1}$ implies the formula.\\
In the second case,  $\|\gamma\smallsetminus e\|=\|\gamma\|$ and
 $\barA(\gamma\to\gamma\smallsetminus e)$ is the identity map.
\end{proof}
\begin{lemma}
\label{L:Dnonred}
Let $\gamma\in\Gamma$ be a non redundant graph. Then the differential $D$ in the total cofibre
\refN{E:TotCofA} satisfies
\[
D\left(s^{-mk}y_\gamma\epsilon_{\|\gamma\|}\right)=
(-1)^{mk}\sum_{e\in\gamma}s^{-mk}
(-1)^{\pos(e:\gamma)}y_{\gamma\smallsetminus e}\Delta(e)\epsilon_{\|\gamma\|+1}.
\]
\end{lemma}
\begin{proof}
Use the formula of $D$ in \refD{totalcofibre} and \refL{Aeps}.
\end{proof}

The following lemma serves to define signs $\lambda(\gamma)=\pm1$
that appears in the definition of $\Phi$ in \refL{defPhi}. The formula below is exactly
the one needed to make $\Phi$ commute with the differential (see \refL{PhiD}.)
\begin{lemma}
\label{L:deflambda}
There exists a  map
$\lambda\colon\Gamma\to\{-1,+1\}$ such that $\lambda(\emptyset)=1$ and
for each non redundant graph $\gamma\in\Gamma$ and $e\in\gamma$
$$\lambda(\gamma)=-(-1)^{m(\pos(e:\gamma)+\|\gamma\|)}\nu(\gamma,e)\lambda(\gamma\smallsetminus e).
$$
\end{lemma}
\begin{proof}
Set $R(\gamma,e):=-(-1)^{m(\pos(e:\gamma)+\|\gamma\|)}\nu(\gamma,e)$
so that  the equation of the statement is 
$\lambda(\gamma)=R(\gamma,e){\cdot}\lambda(\gamma\smallsetminus e)$.
For a non redundant graph $\gamma$ we define $\lambda(\gamma)$ by induction on $|\gamma|$ using this equation
but we need to prove that it is independent of the choice of the edge $e\in\gamma$.
For this it is enough to show that if $e_1$ and $e_2$ are two distinct edges in
 $\gamma$ then
$$
R(\gamma,e_1){\cdot}R(\gamma\smallsetminus e_1,e_2)=R(\gamma,e_2){\cdot}R(\gamma\smallsetminus e_2,e_1),
$$
which is equivalent to 
\begin{equation}
\label{E:nudeflambda}
\nu(\gamma,e_1)\nu(\gamma\smallsetminus e_1,e_2)=(-1)^m
\nu(\gamma,e_2)\nu(\gamma\smallsetminus e_2,e_1).
\end{equation}

Set $r=\|\gamma\|$. Using \refL{Aeps} we compute
\begin{eqnarray*}
(\#\barA(\gamma\smallsetminus e_1\to\gamma\smallsetminus\{e_1,e_2\}))
\left(
\left(\#\barA(\gamma\to\gamma\smallsetminus e_1)\right)(\epsilon_r)\right)=\\
\quad =
\nu(\gamma,e_1)\Delta(e_1)\left(\left( \#A(\gamma\smallsetminus e_1\to\gamma \smallsetminus\{e_1,e_2\} )\right)
(\epsilon_{r+1})\right)&=\\
\quad =
\nu(\gamma,e_1) \Delta(e_1) \nu(\gamma \smallsetminus e_1,e_2) \Delta(e_2) \epsilon_{r+2}.
\end{eqnarray*}
A similar computation gives 
$$
\#\barA(\gamma\smallsetminus e_2\to\gamma\smallsetminus\{e_1,e_2\})
\left(
\#\barA(\gamma\to\gamma\smallsetminus e_2)(\epsilon_r)\right)=
\nu(\gamma,e_2) \Delta(e_2) \nu(\gamma \smallsetminus e_2,e_1) \Delta(e_1) \epsilon_{r+2}.$$
Since $\#\barA$ is a functor, the last two expressions are equal and this implies \refE{nudeflambda}
because $\Delta(e_1)\Delta(e_2)=(-1)^m\Delta(e_2)\Delta(e_1)$.
\end{proof}

\begin{lemma}
\label{L:defPhi}
There exists a unique $A^{\otimes k}$-module map
$$\Phi\colon s^{-mk}\TotCof(\#\barA)\to F(A,k)$$
such that for $\gamma\in\Gamma$
$$\Phi\left(s^{-mk}y_\gamma\epsilon_{\|\gamma\|}\right)=\lambda(\gamma)g_\gamma.
$$
\end{lemma}
\begin{proof}
The factor $s^{-mk}y_\gamma \#A^{\otimes \|\gamma\|}$ is a free 
$A^{\otimes\|\gamma\|}$-module generated by  $s^{-mk}y_\gamma\epsilon_{\|\gamma\|}$.
Its $A^{\otimes k}$-module structure is induced by an algebra map
$A^{\otimes k}\to A^{\otimes \|\gamma\|}$ obtained as a permutation followed by iterated multiplications.
The fact that $\Phi(s^{mk}y_\gamma\epsilon_{\|\gamma\|})=\lambda(\gamma)g_\gamma$ can be extended
to a $A^{\otimes k}$-module map is a consequence of the symmetry relations
$\pi_i^*(a)g_{ij}=\pi_j^*(a)g_{ij}$ in $F(A,k)$.
\end{proof}
Notice that if $\gamma$ is a redundant graph then $\Phi\left(s^{-mk}y_\gamma\epsilon_{\|\gamma\|}\right)=0$.

The three next lemmas establish the equivariance of $\Phi$. 

\begin{lemma}
\label{L:sigmayeps}
Let $\gamma\in\Gamma$ and $\sigma\in\Sigma_k$. We have the following equation in the total cofibre \refN{E:TotCofA}
\[
\sigma{\cdot}\left(s^{-mk}y_\gamma{\cdot}\epsilon_{\|\gamma\|} \right)=
\sgn(\sigma:\gamma)\left(\sgn(\sigma)\sgn(\sigma:\pi_0(\gamma))\right)^m\,
s^{-mk}y_{\sigma{\cdot}\gamma}{\cdot}\epsilon_{\|\sigma{\cdot}\gamma\|}
\]
\end{lemma}
\begin{proof}
The factor $\sgn(\sigma:\gamma)$ is the sign coming from the action on $y_\gamma$
in the cubical diagram as in \refE{gy},
$\sgn(\sigma)^m$ is the orientation-twisting, and
\mbox{$\sgn(\sigma:\pi_0(\gamma))^m$} is the Koszul sign of the permutation
$A^{\otimes\|\gamma\|}\iso A^{\otimes\|\sigma{\cdot}\gamma\|}$
on an element of top degree.
\end{proof}

For $1\leq p\leq k-1$ and for an edge $e\in E$ or a graph $\gamma\in\Gamma$ we set
\begin{equation*}
\eta_e^p:=
\begin{cases} (-1)^m & \text{if $e=(p,p+1)$,}\\
+1 &\text{otherwise.}
\end{cases}\quad
\eta_\gamma^p:=
\begin{cases} (-1)^m & \text{if $(p,p+1)\in\gamma$,}\\
+1 &\text{otherwise.}
\end{cases}
\end{equation*}

\begin{lemma}
\label{L:tau}
Let $1\leq p\leq k-1$, consider the transposition $\tau=(p,p+1)\in\Sigma_k$, let $\gamma\in\Gamma$ be a non redundant graph
and let $e\in\gamma$. Then
\begin{align}
\label{E:taunu}
\nu(\gamma,e)\nu(\tau{\cdot}\gamma,\tau{\cdot}e)&=
\eta^p_e\left(\sgn(\tau:\pi_0(\gamma))\,\sgn(\tau:\pi_0(\gamma\smallsetminus e))\right)^m;
\\
\label{E:taulambda}
\lambda(\gamma)\lambda(\tau{\cdot}\gamma)&=
\eta^p_\gamma\left(-\sgn(\tau:\gamma)\,\sgn(\tau:\pi_0(\gamma))\right)^m;
\\
\label{E:taug}
\tau{\cdot}g_\gamma &=
\eta^p_\gamma\,\sgn(\tau:\gamma)^{m-1}\,g_{\tau{\cdot}\gamma}.
\end{align}
\end{lemma}
\begin{proof}
\refN{E:taunu} Since the differential $D$ on $s^{-mk}\TotCof(\#\barA)$ is equivariant we have
\[\tau.D(s^{-mk}y_\gamma{\cdot}\epsilon_{\|\gamma\|})=
D(\tau.s^{-mk}y_\gamma{\cdot}\epsilon_{\|\gamma\|}).
\]
Develop both sides of this equation using \refL{Dnonred} and \refL{sgngpos}. The sign $\eta^p_e$
comes from the fact that $\tau{\cdot}\Delta((p,p+1))=(-1)^m\Delta((p,p+1))$.

\refN{E:taulambda} By induction on the number of edges $|\gamma|$ using Lemmas \ref{L:deflambda}
 and \ref{L:sgngpos} and the previous formula.
(Hint: in the induction choose the edge $e\in\gamma$ to be $(p,p+1)$ when it belongs to $\gamma$.)

\refN{E:taug} The sign $\eta_\gamma^p$ comes from $g_{p+1,p}=(-1)^m g_{p,p+1}$ and the other sign is the Koszul sign of the
rearrangment of the $g_e$ which are of degree $m-1$.
\end{proof}

\begin{lemma}
\label{L:Phiequiv}
$\Phi$ is $\Sigma_k$-equivariant.
\end{lemma}
\begin{proof}
It is enough to check the equivariance for transpositions $\tau=(p,p+1)$ of adjacent vertices applied
to the generators $s^{-mk}y_\gamma{\cdot}\epsilon_{\|\gamma\|}$.
If $\gamma$ is non redundant it is a computation using Lemmas  \ref{L:sigmayeps} and \ref{L:tau}.
If $\gamma$ is redundant then
the same is true for $\tau{\cdot}\gamma$ and the images by $\Phi$ of the corresponding generators are $0$.
\end{proof}

\begin{lemma}
\label{L:PhiD}
$\Phi$ commutes with the differentials.
\end{lemma}
\begin{proof}
Since $\Phi$ is an $A^{\otimes k}$-module map between $A^{\otimes k}$-DGmodules,
it is enough to check this on the generators $s^{-mk}y_\gamma{\cdot}\epsilon_{\|\gamma\|}$.
For a non redundant graph it is a computation using Lemmas \ref{L:Dnonred} and
\ref{L:deflambda}.
To finish the proof we establish the following:
\\\indent\underline{Claim:}  if $\gamma$ is a redundant graph then $\Phi(D(s^{-mk}y_\gamma\epsilon_{\|\gamma\|}))=0$.\\
 For the sake of the proof we define an $l$-cycle in a graph $\gamma$ as a subset
of edges $\{(i_1,i_2),(i_2,i_3),\ldots,(i_{l-1},i_l),(i_1,i_l)\}$. A graph $\gamma$ is redundant if and only if it contains some $l$-cycle,
with $l\geq 3$, and then $g_\gamma=0$.
Notice that when $\gamma$ contains more than one cycle, in other words
 when $\gamma\smallsetminus e$ is still redundant
for any edge $e\in\gamma$,
 then the claim is obvious. So from now on we suppose that $\gamma$ contains exactly
one cycle.\\
The claim  is easy for the graph $\gamma_{123}:=\{(1,2),(1,3),(2,3)\}$ using the
Arnold relation in $F(A,k)$ (hint: to compare the different  signs $\lambda(\gamma_{123}\smallsetminus e)$,
use \refequ{E:taulambda} in \refL{tau}.) By an induction on the number of edges one deduces the 
claim for any graph containing $\gamma_{123}$ and no other cycle. By the equivariance
of $\Phi$ this implies the result for any graph containing a $3$-cycle.\\
Finally one proves the result for any graph containg an $l$-cycle, for $l\geq4$ by induction on $l$.
Indeed if $\gamma$
contains the $l$-cycle $(1,2),\cdots,(l-1,l),(1,l)$ then consider the graph $\hat\gamma:=\gamma\cup\{(1,3)\}$.
The terms of $D(s^{-mk}y_{\hat\gamma}\epsilon_{\|\hat\gamma\|})$ contains one term
indexed by $\gamma$ and other
terms indexed by a graph containg a cycle of length $<l$ or more than one cycle. 
Using that $D^2=0$ and the inductive hypothesis
one deduces the claim.
\end{proof}
\begin{lemma}
\label{L:Phiqi}
$\Phi$ is  a quasi-isomorphism.
\end{lemma}
\begin{proof}
Let $\Gamma_0\subset\Gamma$ be the subset consisting of graphs of the form
$\{(i_1,j_1),\ldots,(i_l,j_l)\}$ with $1\leq i_1<\cdots<i_l\leq k$ all distinct and $i_s<j_s\leq k$ 
for $s=1,\ldots,l$. Consider the inclusion of chain complexes
\[
\iota\colon s^{-mk}\left(\oplus_{\gamma\in\Gamma_0}
y_\gamma{\cdot}\#(A^{\otimes\|\gamma\|})\right)\quad\hookrightarrow\quad
s^{-mk}\left(\oplus_{\gamma\in\Gamma}
y_\gamma{\cdot}\#(A^{\otimes\|\gamma\|})\right).
\]
An argument completely analogous to that of \cite[Proposition 2.4]{FT:confMassey} (passing to the duals)
shows that $\iota$ is a quasi-isomorphism. Since $\Phi\iota$ is an isomorphism we deduce 
that $\Phi$ is a quasi-isomorphism.
\end{proof}

\section{More general complements and ``all-or-nothing'' transversality}

In summary the idea that we have applied above is that first we
have build a DG-module model of $$C^*(W\smallsetminus \cup_{e\in
E}X_e)$$ of the form of a total cofibre
$$s^{-n}\TotCof(\gamma\mapsto\#C^*(X_\gamma))=s^{-n}\oplus_{\gamma\in\Gamma}
y_\gamma\#C^*(X_\gamma).$$ This only requires a mixture of
Lefschetz duality and a general Mayer-Vietoris principle. The
disadvantage of this model, which works for any system of
subpolyhedra
$$\{X_e\hookrightarrow W\}_{e\in E}$$
is that this model has no clear CDGA structure, partly because
there is no such algebra structure on the duals $\#C^*(X_\gamma)$.

In the case of the configuration space there was another model
which (non-equivariantly at least) is isomorphic to
$$s^{-n}\oplus_{\gamma\in\Gamma_0} \prod_{e\in\gamma}g_e C^*(X_\gamma)$$
which admits a clearer algebra structure. To build this model we
have applied Poincar\'e duality at the cochain level for each of
the submanifold $X_\gamma$,\\
\mbox{ $s^{-n}y_\gamma\#C^*(X_\gamma)\simeq\prod_{e\in\gamma}g_e C^*(X_\gamma)$.} 
 For this to make sense we
need first each of the $X_\gamma$ to be a closed manifold, but
also for all these Poincar\'e dualities at various formal
dimensions  to fit together to recover the Lefschetz duality for
$C^*(W\smallsetminus \cup_{e\in E}X_e)$ we needed some sort of
transversality. In a sense it is exaclty to recover this
transversality that we had to restrict to a subset
$\Gamma_0\subset\Gamma$ defined at the beginning of the proof of \refL{Phiqi}.

In fact this approach can be applied to more general space than
configuration spaces. Actually the main points that we here used
is the fact that we had an oriented manifold $W$ together with a
system of closed submanifolds \mbox{$X_\bullet:=\{X_e\hookrightarrow
W\}_{e\in E}$} such that the families of intersections
$\{\cap_{e\in\gamma}X_e\}_{\gamma\in\Gamma}$ satisfies a certain
``all-or-nothing'' transversality condition that we now
explain.

In the case where $X_\bullet$ is a \emph{total transverse system}
of submanifolds, by which we mean that for any disjoint
$\gamma_1,\gamma_2\subset E$ the submanifolds
$\cap_{e_1\in\gamma_1} X_{e_1}$ and \mbox{$\cap_{e_2\in\gamma_2}
X_{e_2}$} intersects transversally, then using cochain-level Poincar\'e duality 
gives another model of $C^*(W\smallsetminus \cup_{e\in E}X_e)$ of the form
$$(\oplus_{\gamma\in\Gamma} g_\gamma.C^*(X_\gamma),D)$$
where $\deg(g_\gamma)=\codim(X_\gamma)$, and there is a natural
CDGA structure on this when we think to $g_\gamma$ as
$\prod_{e_\in\gamma}g_e$. In other words in the case of a total
transverse system we can take $\Gamma_0=\Gamma$.

In the case of the configuration space we have $E=\{(i,j):1\leq
i<j\leq k \}$ and the familly of diagonals
$X_\bullet:=\{X_{ij}\hookrightarrow M^k\}_{(i,j)\in E}$ is
certainly not totally transverse, except when $k\leq 2$. But it
has another property which we call \emph{all-or-nothing
transverse}. By this we mean that for any
$\gamma_1,\gamma_2\subset E$ the submanifolds
$\cap_{e_1\in\gamma_1} X_{e_1}$ and $\cap_{e_2\in\gamma_2}
X_{e_2}$ either intersects transversally or one of them is
included in the other. This is the case of the system of diagonals
in $M^k$. Using that it is then always possible to find a subset
$\Gamma_0\subset \Gamma$ such that $C^*(W\smallsetminus \cup_{e\in
E}X_e)$ has a model of the form
$$(\oplus_{\gamma\in\Gamma_0} u_\gamma.C^*(X_\gamma),D)$$
where $u_\gamma=\codim(X_\gamma)$, and again this comes with a
natural CDGA structure on this. Actually the subset $\Gamma_0\subset\Gamma$ is characterized by the fact
that if $\gamma\in\Gamma_0$ and $\gamma'\subset\gamma$ then $\gamma'\in\Gamma_0$ 
and, for $e\not\in\gamma$, we have $\gamma\cup\{e\}\in\Gamma_0$ if and only if 
$X_{\gamma\cup\{e\}}'\not= X_\gamma$. We do not claim that it is a CDGA
model in general, and finding suitable conditions for this to be
true is certainly an interesting but difficult problem.

This approach could be usefull to the study of other complements,
like systems of projective subspaces in $\BC P(n)$ but we will not develop this
further in this paper.

\def\cprime{$'$}

\end{document}